\documentclass{amsart}

\usepackage[utf8]{inputenc}
\usepackage{amsmath,amsfonts,amssymb,amsthm}
\usepackage{color}
\usepackage{epsdice}

\DeclareMathOperator{\Int}{int}

\newcommand{\tnij}{\, \sqcap \,}

\newcommand{\R}{\mathbb{R}}
\newcommand{\mm}{\mathrm{m}}
\DeclareMathOperator{\inter}{int}
\DeclareMathOperator{\bd}{bd}
\newcommand{\edges}{E}
\newcommand{\edgesp}{E^+}
\newcommand{\nbhd}{N}
\renewcommand{\vartheta}{\Box}

\newtheorem{claim}{Claim}
\newtheorem{prop}[claim]{Proposition}

\newtheorem{theorem}[claim]{Theorem}

\newtheorem{lemma}[claim]{Lemma}
\newtheorem{definition}[claim]{Definition}
\newtheorem{corollary}[claim]{Corollary}
\newtheorem{question}[claim]{Question}

\newcommand{\Z}{\mathbb{Z}}
\newcommand{\N}{\mathbb{N}}

\DeclareMathOperator{\dom}{dom}
\DeclareMathOperator{\im}{im}
\DeclareMathOperator{\supp}{supp}

\author{Tomasz Cie\'{s}la}
\author{Marcin Sabok}

\thanks{This research was partially supported by the NSERC
  through the \textit{Discovery Grant} RGPIN-2015-03738, by the FRQNT
  (Fonds de recherche du Qu\'{e}bec) grant
  \textit{Nouveaux
    chercheurs} 2018-NC-205427 and
  by the NCN (Polish National Science Centre) through the
  grants \textit{Harmonia} no. 2015/18/M/ST1/00050 and 2018/30/M/ST1/00668}

\address{Department of Mathematics and Statistics, McGill
  University, 805, Sherbrooke Street West Montreal, Quebec,
  Canada H3A 2K6}
\address{Institute of Mathematics, Polish
  Academy of Sciences, \'Sniadeckich 8, 00-655 Warszawa,
  Poland} 
\email{marcin.sabok@mcgill.ca}

\address{Department of Mathematics and Statistics, McGill
  University, 805, Sherbrooke Street West Montreal, Quebec,
  Canada H3A 2K6}
\email{tomasz.ciesla@mail.mcgill.ca}

\title{Measurable Hall's theorem for actions of abelian groups}

\begin{document}
\maketitle
\begin{abstract}
  We prove a measurable version of the Hall marriage theorem
  for actions of finitely generated abelian groups. In particular, it implies
  that for free measure-preserving actions of such groups
  and measurable sets which are suitably equidistributed with respect to the
  action, if they are
  are equidecomposable, then they are equidecomposable
  using measurable pieces. The latter generalizes a recent
  result of Grabowski, M\'ath\'e and Pikhurko on the measurable
  circle squaring and confirms a special case of a conjecture
  of Gardner.
\end{abstract}

\section{Introduction}
\label{sec:introduction}

In 1925 Tarski famously asked if the unit square and the
disk of the same area are equidecomposable by isometries of
the plane, i.e. if one can partition one of them into
finitely many pieces, rearrange them by isometries and
obtain the second one. This problem became known as the
Tarski circle squaring problem.

The question whether two sets of the same measure can be
partitioned into congruent pieces has a long history. At the
beginning of the 19th century Wallace, Bolyai and Gerwien
showed that any two polygons in the plane of the same area are
congruent by dissections (see \cite[Theorem 3.2]{wagon}) and
Tarski \cite{tarski.rownowaznosc} (\cite[Theorem
3.9]{wagon}) showed that such polygons are equidecomposable
using pieces which are polygons themselves. Hilbert's 3rd
problem asked if any two polyhedra of the same volume are
equidecomposable using polyhedral pieces. The latter was
solved by Dehn (see \cite{dehn}). Banach and Tarski showed
that in dimension at least 3, any two bounded sets in $\R^n$
with nonempty interior, are equidecomposable, which leads to
the famous Banach--Tarski paradox on doubling the
ball. Back in dimension $2$, the situation is somewhat
different, as any two measurable subsets equidecomposable by
isometries must have the same measure (see \cite{wagon}) and
this was one of the motivation for the Tarski circle
squaring problem.  Using isometries was also essential as
von Neumann \cite{vn} showed that the answer is positive if
one allows arbitrary area-preserving transformations.  The
crucial feature that makes the isometries of the plane
special is the fact that the group of isometries of $\R^2$
is amenable. Amenability was, in fact, introduced by von
Neumann in the search of a combinatorial explanation of the
Banach--Tarski paradox.

The first partial result on the Tarski circle squaring was a
negative result of Dubins, Hirsch and Karush \cite{dhk} who
showed that pieces of such decompositions cannot have smooth
boundary (which means that this cannot be performed using
scissors). However, the full positive answer was given by
Laczkovich in his deep paper \cite{laczkovich}. In fact, in
\cite{laczkovich.general} Laczkovich proved a stronger
result saying that whenever $A$ and $B$ are two bounded
measurable subsets of $\R^n$ of positive measure such that
the upper box
dimension of the boundaries of $A$ and $B$ is less than
$n$, then $A$ and $B$ are equidecomposable. The assumption
on the boundary is essential since Laczkovich
\cite{laczkovich.boundary} (see also
\cite{laczkovich.jordan}) found examples of two measurable
sets of the same area which are not equidecomposable even
though their boundaries have even the same Hausdorff
dimension. The proof of Laczkovich, however, did not provide
any regularity conditions on the pieces used in the
decompositions. Given the assumption that $A$ and $B$ have
the same measure, it was natural to ask if the pieces can be
chosen to be measurable. Moreover, the proof of Laczkovich
used the axiom of choice.

A major breakthrough was achieved recently by Grabowski,
M\'ath\'e and Pikhurko \cite{gmp} who showed that the pieces
in Laczkovich's theorem can be chosen to be measurable:
whenever $A$ and $B$ are two bounded subsets of $\R^n$ of
positive measure such that the upper box dimension of the
boundaries of $A$ and $B$ are less than $n$, then $A$ and
$B$ are equidecomposable using measurable pieces. Another
breakthrough came even more recently when Marks and Unger
\cite{mu} showed that for Borel sets, the pieces in the
decomposition can be even chosen to be Borel, and their
proof did not use the axiom of choice.

The goal of the present paper is to give a combinatorial
explanation of these phenomena. There are some limitations
on how far this can go because already in Laczkovich's
theorem there is a restriction on the boundary of the sets
$A$ and $B$. Therefore, we are going to work in the
measure-theoretic context and provide sufficient and
necessary conditions for two sets to be equidecomposable
almost everywhere. Recently, there has been a lot of effort
to develop methods of the measurable and Borel
combinatorics (see for instance the upcoming monograph by
Marks and Kechris \cite{marks.kechris}) and we would like to
work within this framework.

The classical Hall marriage theorem provides sufficient and
necessary conditions for a bipartite graph to have a perfect
matching. Matchings are closely connected with the existence
of equidecompositions and both have been studied in this
context. In 1996 Miller \cite[Problem 15.10]{miller} asked
whether there exists a Borel version of the Hall
theorem. The question posed in such generality has a
negative answer as there are examples of Borel graphs which
admit perfect matchings but do not admit measurable perfect
matchings. One example is provided already by the
Banach-Tarski paradox (see \cite{marks.kechris}) and Laczkovich
\cite{laczkovich.matchings} constructed a closed graph which
admits a perfect matching but does not have a measurable
one. In the Baire category setting, Marks and Unger
\cite{marks.unger} proved that if a bipartite Borel graph
satisfies a stronger version of Hall's condition with an
additional $\varepsilon>0$, i.e. if the set of neighbours of
a finite set $F$ is bounded from below by
$(1+\varepsilon)|F|$, then the graph admits a perfect
matching with the Baire property (see also \cite{marks.det}
and \cite{conley.miller} for related results on matchings in
this context). On the other hand, in all the results of
Laczkovich \cite{laczkovich.general}, Grabowski, M\'ath\'e
and Pikhurko \cite{gmp} and Marks and Unger \cite{mu} on the
circle squaring, a crucial role is played by the strong
discrepancy estimates, with an $\varepsilon>0$ such that the
discrepancies of both sets are bounded by
$C\frac{1}{n^{1+\varepsilon}}$ (for definitions see Section
\ref{sec:discrepancies}).  Recall that given a finitely
generated group $\Gamma$ generated by a symmetric set $S$
and acting freely on a space $X$, the \textit{Schreier
  graph} of the action is the graph connecting two points
$x$ and $y$ if $\gamma\cdot x=y$ for one of the generators
$\gamma\in S$.

\begin{definition}
  Suppose $\Gamma\curvearrowright (X,\mu)$ is a free pmp
  action of a finitely generated group on a space $X$. Write
  $G$ for the Schreier graph of the action. A pair of sets
  $A,B$ satisfies the \textit{Hall condition} ($\mu$-a.e.)
  with respect to $\Gamma$ (given a set of generators) 
  if for every ($\mu$-a.e.) $x\in X$ and for every finite subset $F$ of
  $\Gamma\cdot x$ we have
  $$|F\cap A|\leq|N_G(F)\cap B|,\quad |F\cap
  B|\leq|N_G(F)\cap A|,$$
  where $N_G(F)$ means the neighborhood of $F$ in the graph
  $G$.
\end{definition}

This definition clearly depends on the choice of generators,
and we say that $A,B$ satisfy the \textit{Hall
  condition} ($\mu$-a.e.) if the above holds for some choice
of generators. For the case with a fixed set of generators
(which will be more natural for us),
we say that the action $\Gamma\curvearrowright (X,\mu)$
satisfies \textit{$k$-Hall condition} ($\mu$-a.e.) if for
every ($\mu$-a.e., resp.)  $x\in X$ for every finite subset
$F$ of $\Gamma\cdot x$ we have
$$|F\cap A|\leq|N^k_G(F)\cap B|,\quad |F\cap B|\leq|N^k_G(F)\cap A|,$$ where
$N^k_G(F)$ denotes the $k$-neighborhood of $F$ in the
graph $G$. Note that $A,B$ satisfy the Hall condition if and
only if $A,B$ satisfy the $k$-Hall condition for some
$k>0$. 
We
will work under the assumption that both sets $A,B$ satisfy
certain form of
equidistribution on the orbits, namely that they are \textit{$\Gamma$-uniform} (for definition see Section
\ref{sec:discrepancies}).

Our main result is the following. 

\begin{theorem}\label{hall}
  Let $\Gamma$ be a finitely generated abelian group and
  $\Gamma\curvearrowright (X,\mu)$ be a free pmp action. Suppose
  $A,B\subseteq X$ are two measurable $\Gamma$-uniform
  sets. The following are equivalent:
  \begin{enumerate}
  \item the pair $A,B$ satisfies the Hall condition with respect
    to $\Gamma$ $\mu$-a.e.,
  \item $A$ and $B$ are $\Gamma$-equidecomposable
    $\mu$-a.e. using $\mu$-measurable sets,
  \item $A$ and $B$ are $\Gamma$-equidecomposable $\mu$-a.e.
  \end{enumerate}
\end{theorem}

As a consequence, it gives the following.

\begin{corollary}\label{wniosek}
  Suppose $\Gamma$ is a finitely generated  abelian
  group and $\Gamma\curvearrowright(X,\mu)$ is a free pmp Borel
  action on a standard Borel probability space. Let
  $A,B\subseteq X$ be measurable $\Gamma$-uniform
  sets. If $A$ and $B$ are $\Gamma$-equidecomposable, then
  $A$ and $B$ are $\Gamma$-equidecomposable using measurable
  pieces.
\end{corollary}

This generalizes the recent measurable circle squaring
result \cite{gmp} as Laczkovich \cite{laczkovich}
constructs an action of $\Z^d$ satisfying the conditions
above, for a suitably chosen $d$ (big enough, depending on
the box dimensions of the boundaries).

In fact, in 1991 Gardner \cite[Conjecture 6]{gardner}
conjectured that whenever two bounded measurable sets in an
Euclidean space are equidecomposable via an action of an
amenable group of isometries, then they are equidecomposable
using measurable sets. The above corollary confirms this
conjecture in case of an abelian group $\Gamma$ and
$\Gamma$-uniform sets.

\medskip

The main new idea in this paper is an application of
measurable medial means. These have previously used in
descriptive set theory only in the work of Jackson, Kechris
and Louveau \cite{jkl} on amenable equivalence relations but that
context was combinatorially different.\footnote{As we have
  learnt recently, a similar idea can be also found in
  \cite{lacz.meas,wehrung}.} They are used together with a
recent result of Conley, Jackson, Kerr, Marks, Seward and
Tucker-Drob \cite{6authors} on tilings of amenable group
actions in averaging sequences of measurable matchings. This
allows us to avoid using Laczkovich's discrepancy estimates
that play a crucial role in both proofs of the measurable
and Borel circle squaring. We also employ the idea of Marks
and Unger in constructing bounded measurable flows. More
precisely, following Marks and Unger we construct bounded
integer-valued measurable flows from bounded real-valued
measurable flows. However, instead of using Tim\'ar's result
\cite{timar} for specific graphs induced by actions of
$\Z^d$, we give a self-contained simple proof of the latter
result, which works in the measurable setting for the
natural Cayley graph of $\Z^d$. This is the only part of the
paper which deals with abelian groups and we hope it could
be generalized to a more general setting. On the other hand,
the measurable averaging operators that we employ, cannot be
made Borel and for this reason the results of this paper
apply to the measurable setting and generalize only the
results of \cite{gmp}.

\medskip


While this paper deals with abelian groups (the crucial and
only place which works under these assumptions is Section
\ref{sec:flows}), a positive answer to the following
question would confirm Gardner's conjecture \cite[Conjecture
6]{gardner}. \footnote{This has been recently answered in the
    negative by Kun \cite{gabor}, and the recent preprint
    \cite{bks} proves some optimal results the positive direction.}

\begin{question}
  Is the measurable version of Hall's theorem true for free
  pmp actions of finitely generated amenable
  groups? 
\end{question}

\medskip
\noindent\textbf{Acknowledgements} We are grateful to Oleg Pikhurko for useful
comments on an early version of the manuscript. We also
thank the anonymous Referee for many helpful remarks.

\section{Discrepancy estimates}
\label{sec:discrepancies}

Both proofs of Grabowski, M\'ath\'e and Pikhurko and of Marks
and Unger use a technique that appears in Laczkovich's paper
\cite{laczkovich} and is based on discrepancy
estimates. Laczkovich constructs an action of a group of the
form $\Z^d$ for $d$ depending on the upper box dimension of the
boundaries of the sets $A$ and $B$ such that both sets are
very well equidistributed on orbits on this action. To be
more precise, given an action $\Z^d\curvearrowright (X,\mu)$
and a measurable set $A\subseteq X$, the
\textit{discrepancy} of $A$ with respect to a finite subset
$F$ of an orbit of the action is defined
as $$D(F,A)=
\left|\frac{|A\cap F|}{|F|}-\mu(A)\right|.$$

It is meaningful to compute the discrepancy with respect to
finite cubes, i.e. subsets of orbits which are of the form
$[0,n]^d\cdot x$, where $x\in X$ and $[0,n]^d\subseteq\Z^d$
is the $d$-dimensional cube with side $\{0,\ldots,n\}$. The
cube $[0,n]^d$ has boundary, whose relative size with
respect to the size of the cube is bounded by $c\cdot\frac{1}{n}$ for a
constant $c$. 

A crucial estimation that appears in Laczkovich's paper is
that the action of $\Z^d$ is such that for both sets $A$ and
$B$ the discrepancy is estimated as
\begin{equation}D([0,n]^d\cdot x,A),D([0,n]^d\cdot x,B)\leq
c\frac{1}{n^{1+\varepsilon}},\label{esstimates}\tag{$*$}
\end{equation}
for some $\varepsilon>0$ and some $c>0$, which means that the discrepancies
of both sets on cubes decay noticeably faster than the sizes
of the boundaries of these cubes. 

A slightly more natural condition on the equidistrubition of
a set $A$ would be to remove the
$\varepsilon$ from (\ref{esstimates}) and require that
there exists a constant $c$ such that for every $n$ the
discrepancy
\begin{equation}
D([0,n]^d\times\Delta\cdot x,A)\leq c\frac{1}{n}\label{essstimates}\tag{$**$}
\end{equation}
for $\mu$-a.e. $x$.  In fact, as shown in \cite[Theorem
1.2]{laczkovich.uniform}, the condition (\ref{esstimates})
implies that (\ref{essstimates}) is satisfied on every union
of cubes in places of $[0,n]^d$. Sets satisfying the latter
condition are called \textit{uniformly spread} (see
\cite[Theorem 1.1]{laczkovich.uniform}). However, in
\cite[Theorem 1.5]{laczkovich.uniform}, Laczkovich gives
examples of sets which satisfy (\ref{essstimates}) but are not
uniformly spread.

In this paper, we work with even weaker assumption on equidistribution, given by
the following definition.

\begin{definition}
  Given a Borel free pmp action $\Gamma\curvearrowright
  (X,\mu)$ of a finitely generated
  abelian group 
  $\Gamma=\Z^d\times\Delta$ with $\Delta$
  finite, and a measurable set $A\subseteq X$, we say that
  $A$ is \textit{$\Gamma$-uniform} if there exists a constant
  $c>0$ such that for $\mu$-a.e. $X$ we have $$|A\cap (F\cdot x)|\geq c
  |F|$$ whenever $F$ is a set of the form $F=[0,n]^d\times\Delta$.
\end{definition}

Note that this definition does not depend (up to changing
the constant $c$) on the way the group is
written as $\Z^d\times\Delta$ and a choice of generators for
the group. 


\section{Measurable averaging operators}
\label{sec:absoluteness}

In this paper we use special kinds of measurable averaging
operators. These can be constructed in different ways.

For the first construction, recall the definition of a medial mean.
\begin{definition}
  A \textit{medial mean} is a linear functional
  $\mm:\ell_\infty\to\R$ which is positive, i.e.
  $\mm(f)\geq 0$ if $f\geq 0$, normalized,
  i.e. $\mm(1_\N)=1$ and shift invariant,
  i.e. $\mm(Sf)=\mm(f)$ where $Sf(n)=f(n+1)$.
\end{definition}

Medial means were studied already by Banach who showed their
existence (the so-called \textit{Banach limits}). In this
paper we use a special kind of medial means $\mm$ which is
additionally measurable on $[0,1]^\mathbb{N}$ and we use it to take a measurable
average of a sequence of functions $f_n:(X,\mu)\to[0,1]$ for a
space $(X,\mu)$ with a Borel probability measure $\mu$.
It takes a bit more effort to construct medial means that
are measurable but this can be done in a couyple of ways.

Recall that Mokobodzki showed that under the assumption of
the Continuum Hypothesis there exists a medial mean which is
universally measurable as a function on $[0,1]^\N$. For a
proof the reader can consult the textbook of Fremlin
\cite[Theorem 538S]{fremlin} or the article
\cite{meyer}. However, a careful analysis of Mokobodzki's
proof shows that for a single Borel probability measure
$\mu$, the existence of a $\mu$-measurable medial mean does
not require the Continuum Hypothesis. Nevertheless, some
set-theoretical assumptions are still used even for a single
$\mu$, such as the Hahn--Banach theorem.

Another construction of measurable averaging a sequence of
measurable functions $f_n:X\to[0,1]$ (perhaps more familiar
to the general mathematical audience than the Mokobodzki
construction) can be done using the Banach--Saks theorem in
$L^2(X,\mu)$ by using weak${}^*$-compactness of the unit
ball (cf. \cite{lacz.meas}). This also needs some weak form
of the Axiom of Choice such as the Tichonov theorem.

\section{Set-theoretical assumptions}
\label{sec:set-theor-assumpt}

In view of the foundational questions and the role of the
Axiom of Choice in equidecompositions (e.g. recall
that the Hahn--Banach theorem implies the Banach--Tarski
paradox \cite{pawlikowski}), we argue below that for the measurable equidecompositions
on co-null sets, we can remove any set theoretic assumptions
beyond ZF and the Axiom of Dependent Choice (DC)
that are needed to obtain a measurable medial mean.

Recall that Borel sets can be coded using a
$\mathbf{\Pi}^1_1$ set (of Borel codes)
$\mathrm{BC}\subseteq 2^\N$ in a $\mathbf{\Delta}^1_1$ way,
i.e. there exists a subset $C\subseteq\mathrm{BC}\times X$
such that the family $\{ C_x: x\in\mathrm{BC}\}$ consists of
all Borel subsets of $X$ and the set $C$ can be defined
using both $\mathbf{\Sigma}^1_1$ and $\mathbf{\Pi}^1_1$
definitions. For details the reader can consult the textbook
of Jech \cite[Chapter 25]{jech}.

Given a Borel probability measure $\mu$ on $X$ and a subset
$P\subseteq X\times Y$, we write $\forall^\mu\ x P(x,y)$ to
denote that $\mu(\{x\in X: P(x,y)\})=1$. It is well known
\cite[Chapter 29E]{kechris} that if $P$ is
$\mathbf{\Sigma}^1_1$, then
$\{y\in Y: \forall^\mu x\ P(x,y)\}$ is
$\mathbf{\Sigma}^1_1$.

The proposition below implies that if two Borel sets are
coded using a real $r$, then we can argue about their
measurable equidecomposition $a.e.$ in $L[r]$ (where AC and
CH hold).

\begin{prop}
  Let $V\subseteq W$ be models of \textup{ZF+DC}. Suppose in
  $V$ we have a standard Borel space $X$ with a Borel
  probability measure $\mu$, two Borel subsets
  $A,B\subseteq X$ and
  $\Gamma\curvearrowright (X,\mu)$ is a Borel pmp action of a
  countable group $\Gamma$. The statement that the sets $A$ and $B$ are
  $\Gamma$-equidecomposable $\mu$-a.e. using
  $\mu$-measurable pieces is absolute between $V$ and $W$.
\end{prop}
\begin{proof}

  Suppose that in $W$ or $V$ the sets $A$ and $B$ are
  $\Gamma$-equidecomposable $\mu$-a.e. Then there exist
  disjoint Borel subsets $A_1,\ldots A_n$ of $A$ and
  disjoint Borel subsets $B_1,\ldots,B_n$ of $B$ such that
  $\mu(A\setminus\bigcup_{i=1}^n A_i)=0$,
  $\mu(B\setminus\bigcup_{i=1}^n B_i)=0$ and
  $\gamma_i A_i=B_i$ for some
  $\gamma_1,\ldots,\gamma_n\in \Gamma$.  This statement can
  be written as
  \begin{align*}
    \exists x_1,\ldots,x_n\ 
    \bigwedge_{i\leq n}\mbox{BC}(x_i)\wedge \bigwedge_{i\not=j} C_{x_i}\cap
    C_{x_j}=\emptyset\\ \wedge\ \forall^{\mu} x\  (x\in
    A\leftrightarrow \bigvee_{i=1}^n x\in C_{x_i}) 
    \wedge\ \forall^{\mu} x\  (x\in
    B\leftrightarrow \bigvee_{i=1}^n x\in \gamma_i C_{x_i})
  \end{align*}
  and thus is it $\mathbf{\Sigma}^1_2$. By Shoenfield's
  absoluteness theorem \cite[Theorem 25.20]{jech}, it is
  absolute between $V$ and $W$.
\end{proof}



\section{Measurable flows in actions of amenable groups}
\label{sec:meas-flows-acti}

Given a standard Borel space $X$, a Borel graph $G$
on $X$ and $f:X\to\R$, a function $\varphi:G\to \R$ is an
\textit{$f$-flow} if $\varphi(x,y)=-\varphi(y,x)$ for every
$(x,y)\in G$ and $f(x)=\sum_{(x,y)\in G}\varphi(x,y)$ for
every $x\in X$. 

Let $\Gamma$ be a finitely generated
amenable group. Let $\gamma_1, \ldots, \gamma_d$ be a finite
symmetric set of generators of $\Gamma$. Let $X$ be a standard
Borel space and let $\mu$ be a Borel probability measure on
$X$. Let $\Gamma \curvearrowright (X,\mu)$ be a free pmp
action. Recall that by the \textit{Schreier graph} of the action we mean
the graph
$\{(x,\gamma_ix) : x \in X, 1 \le i \le d \} \subseteq X
\times X$.

\begin{definition}
For finite sets $F,K \subseteq \Gamma$ and $\delta>0$ we say that
$F$ is \textit{$(K,\delta)$-invariant} if
$|KF \triangle F| < \delta|F|$.
\end{definition}


In the following lemma we assume that there exists a
universally measurable medial mean $\mm$, which, by the
remarks in the previous section, we can assume throughout
this paper.

In order to make it a bit more general, let us define the
Hall condition for functions: a function $f:X\to\Z$
satisfies the \textit{$k$-Hall condition} if for every finite set $F$
contained in an orbit of $\Gamma\curvearrowright X$ we have
that $$\sum_{x\in F: f(x)\geq0}f(x)\leq
\sum_{x\in N^k_G(F):f(x)\leq0}-f(x),\ \sum_{x\in F:
  f(x)\leq0}-f(x)\leq \sum_{x\in N^k_G(F):f(x)\geq0}f(x).$$
Note that a pair of sets $A,B$ satisfies the $k$-Hall
condition if and only if  $f=\chi_A-\chi_B$ satisfies the
$k$-Hall condition.

\begin{prop}\label{real}
   Let $\Gamma$ be a finitely generated amenable group and
   $\Gamma\curvearrowright (X,\mu)$ be a Borel free pmp
   action. Suppose $f:X\to\Z$ is a measurable
   function such that
   \begin{itemize}
   \item $|f|\leq l$
   \item $f$ satisfies the $k$-Hall condition
   \end{itemize}
   for some $k,l\in\N$. Then there
   exists a $\Gamma$-invariant measurable subset $X'\subseteq X$ of measure $1$ and
   a measurable real-valued $f$-flow $\phi$ on the Schreier graph of
   $\Gamma\curvearrowright X'$ such that $$|\phi|\leq l\cdot
   d^k,$$ where $d$ is the number of generators of $\Gamma$.
\end{prop}
\begin{proof}
  First, we are going to assume that $|f|\leq 1$, i.e. that
  $f=\chi_A-\chi_B$ for two measurable subsets
  $A,B\subseteq X$.  Indeed, replace $X$ with $X\times l$
  and take the projection $\pi: X\times l\to X$. Then we can
  find two subsets $A,B\subseteq X\times l$ such that
  $f(x)=|\pi^{-1}(\{x\})\cap A|-|\pi^{-1}(\{x\})\cap B|$. We
  can also induce the graph structure on $X\times l$ by
  taking as edges all the pairs $((x,i),(y,j))$ such that
  $(x,y)$ forms an edge in $X$ as well as all pairs
  $((x,i),(x,j))$ for $i\not=j$. Then $A$ and $B$ satisfy
  the $k$-Hall condition in $X\times l$ for the above graph.

   Let $K=\{\gamma\in \Gamma \colon d(e,\gamma)\le k\}$. Fix
   $\delta>0$. Use the
   Conley--Jackson--Kerr--Marks--Seward--Tucker-Drob tiling
   theorem \cite[Theorem 3.6]{6authors}
   for $K$ and $\delta$ to get a $\mu$-conull
   $\Gamma$-invariant Borel set $X' \subseteq X$, a collection
   $\{C_{i} : 1 \le i \le m\}$ of Borel subsets of $X'$, and
   a collection $\{F_{i} : 1 \le i \le m\}$ of
   $(K,\delta)$-invariant subsets of $\Gamma$ such that
   $\mathcal{F} = \{F_{i}c : 1 \le i \le m, c \in C_{i}\}$
   partitions $X'$.

   For a finite set $F \subseteq \Gamma$ define
   $F(K) = \{f\in F : Kf \subseteq F\}$.  Note that if
   $F'x = F''y$ where $F',F''$ are finite subsets of
   $\Gamma$ and $x, y \in X$ then $F'(K)x = F''(K)y$. If
   $F\subseteq X$ is a finite subset of a single orbit, then
   we let $F(K) = F'(K)x$ where $F' \subseteq \Gamma$ and
   $x \in X$ satisfy $F = F'x$. This definition does not
   depend on the choice of representation $F = F'x$ by the
   previous remark.
Note that if
  $F \subseteq X$ is $(K,\delta)$-invariant then
  $$|F(K)| \geq |F| - |KF\triangle F| \cdot |K| > |F| \cdot (1 - \delta|K|).$$

  Write
  $$H = \{ (x,\gamma x) \in A \times B \colon x \in F_{i}(K)\cdot c
  \text{ for some } 1\le i \le m \text{ and } c \in C_{i}, \
  \gamma \in K\}.$$
  Then $H$ is a locally finite Borel graph satisfying Hall's
  condition as $A,B$ satisfy the $k$-Hall condition. By the Hall theorem, there
  exists a Borel injection

  $$h \colon A \cap \bigcup_{F \in \mathcal{F}} F(K) \to B \cap \bigcup_{F \in \mathcal{F}} F.$$
  Write $G$ for the Schreier graph of
  $\Gamma\curvearrowright X$. For every $x \in \dom h$ let
  $p_{x} = \{(x_0, x_1), (x_1, x_2), \ldots, (x_{j-1},
  x_j)\}$
  be the shortest lexicographically smallest path in the
  graph $G$ connecting $x_0=x$ with $x_j=h(x)$. Let
  $\mathcal P = \{ p_{x} \colon x \in \dom h\}$.

  Define $\phi \colon G \to \R$ by the formula
  $$\phi(x,\gamma x) = |\{p \in \mathcal P \colon
  (x,\gamma x) \in p \}| - |\{p\in\mathcal P \colon (\gamma x,x) \in p\}|.$$

  Note that $\phi$ is Borel (by definition). Also, $|\phi|$
  is bounded by $d^k$ (the number of paths of length not greater
  than $k$ passing through a given edge in the 
  graph $G$). By
  definition, $\phi$ is a
  $(\chi_{\dom h} - \chi_{\im h})$-flow.

  Define
  $$X''=\bigcup_{F \in \mathcal{F}} F(K) \setminus (B \setminus h(A)).$$
  Note that for every $x \in X''$ we have
  $\chi_A(x)-\chi_B(x) = \chi_{\dom h}(x) - \chi_{\im h}(x)$.

  For every $1\le i \le m$ let
  $\{(A_{1,i},B_{1,i},h_{1,i}), (A_{2,i},B_{2,i},h_{2,i}),
  \ldots, (A_{n_i,i},B_{n_i,i},h_{n_i,i})\}$
  be the set of all triples $(A',B',h')$ consisting of sets
  $A',B'\subseteq F_i$ and a bijection $h'\colon A' \to
  B'$.
  For $1 \le j \le n_i$ define
  $$C_{j,i} = \{c \in C_i \colon (\dom h_{j,i})c = A \cap
  (F_ic) \wedge\, \forall \gamma \in \dom h_{j,i} \
  h_{j,i}(\gamma )c=h(\gamma c)
  \}.$$
  Then $\{C_{1,i}, C_{2,i}, \ldots, C_{n_i,i}\}$ is a
  partition of $C_i$ into Borel sets.

  Observe that for every $F \in \mathcal F$ we have
  $$|h(A) \cap F(K)| \ge |F(K) \cap A| - |F \setminus
  F(K)|$$
  and
  $$|B \cap F(K)| \le |A \cap F| \le |A \cap F(K)| +
  |F \setminus F(K)|.$$ Therefore
  \begin{align*}
    |F(K) \cap (B \setminus h(A))| &= |(F(K) \cap B) \setminus
                                     (F(K) \cap h(A))| \\ &=|F(K) \cap
                                                            B| - |F(K) \cap h(A)| \\ &\le |A \cap F(K)| + |F
                                                                                       \setminus F(K)| - (|F(K) \cap A| - |F \setminus
                                                                                       F(K)|) \\ &= 2|F\setminus F(K)|.
  \end{align*}
  It follows that
  \begin{align*}
|F(K) \setminus (B
    \setminus h(A))| =|F(K)|-|F(K)\cap(B\setminus h(A))| \\ \ge
    |F(K)| - 2|F \setminus F(K)| = 3|F(K)|-2|F| >
    |F|(1-3\delta |K|).
  \end{align*}
  Therefore
  \begin{align*}
    \mu(X'') &= \mu(\bigcup_{i=1}^{m} \bigcup_{j=1}^{n_i} ( F_{i}(K)C_{j,i} \setminus (B \setminus h(A)))) = \sum_{i=1}^m \sum_{j=i}^{n_i} |F_{i}(K)\setminus(B_{j,i} \setminus A_{j,i})| \mu(C_{j,i}) \\
             &> \sum_{i=1}^m \sum_{j=i}^{n_i}
               |F_{i}|(1-3\delta|K|)\mu(C_{j,i}) = \sum_{i=1}^m
               |F_i|(1-3\delta|K|)\mu(C_i) \\ &= 1-3\delta|K|.
  \end{align*}

  Now, for every $n$ pick $\delta_n>0$ so that
  $1-3\delta_n|K|>1-\frac{1}{2^n}$. Denote $h_n=h$,
  $\phi_n=\phi$ and $X_n=X''$ where $h$, $\phi$ and $X''$
  are constructed above for this particular $\delta_n$.

  Let
  $Y=\liminf X_{n} = \bigcup_{m=1}^\infty
  \bigcap_{n=m}^\infty X_{n}$.  Then $\mu(Y)=1$. We can
  assume that $Y$ is $\Gamma$-invariant (by taking its
  subset if needed). Denote by $G$ the
  Schreier graph of $\Gamma\curvearrowright Y$. Write
  $\phi_\infty=(\phi_{n})_{n\in\N}: G\to \ell^\infty$. Define
  $$\phi(x,y)=\mm(\phi_\infty(x,y)),$$ where $\mm$ denotes
  the medial mean. Then for $x \in Y$ we have

  \begin{align*}
    \sum_{y \colon (x,y) \in G} \phi(x,y) &= \sum_{y \colon (x,y) \in G} \mm((\phi_{n}(x,y))_{n\in\N}) = \mm((\sum_{y \colon (x,y) \in G} \phi_{n}(x,y))_{n\in\N}) \\
                                          &= \mm((\chi_{\dom h_{n}}(x) - \chi_{\im h_{n}}(x))_{n\in\N})=\chi_A(x)-\chi_B(x)
  \end{align*}
  as the sequence
  $\chi_{\dom h_{n}}(x) - \chi_{\im h_{n}}(x)$ is eventually
  constant and equal to $\chi_A(x)-\chi_B(x)$.

  Therefore $\phi$ is a $(\chi_A - \chi_B)$-flow in the
  Schreier graph $G$ of $\Gamma\curvearrowright Y$. Moreover, $|\phi|$
  is bounded by $d^k$, which is a common bound for the flows
  $\phi_{n}$. For measurability of $\phi$, write
  $\mu'=\phi_*(\mu\times\mu)$ for the pushforward  to $[-d^k,d^k]^\N$ of the measure
  $\mu\times\mu$ on the graph $G$ and note that since $\mm$ is
  $\mu'$-measurable, it follows that $\phi$ is
  $\mu$-measurable.

\end{proof}

\section{Flows in $\mathbb{Z}^d$}
\label{sec:flows}

In this section we prove a couple of combinatorial lemmas
which lead to a finitary procedure of changing a 
real-valued flow on a cube in $\Z^d$ to an integer-valued
flow on a cube in $\Z^d$. This gives an alternative proof of
\cite[Lemma 5.4]{mu} in the measurable setting. Also, this is the only
part of the paper which deals with the groups $\Z^d$ as
opposed to arbitrary amenable groups.

Let
$$G = \{(x,x') \in \Z^d \times \Z^d \colon x'-x \in \{\pm e_1, \pm e_2, \ldots, \pm e_d\} \}$$
be the Cayley graph of $\Z^d$. An edge $(x,x')$ is called
\textit{positively oriented} if $x'-x=e_j$ for some $j$.

\begin{definition}
For a set $A \subseteq \Z^d$ we define:
$$\edges(A)=\{(x,x+e_j) \colon j \in \{1,2,\ldots,d\}, \ \{x,x+e_j\}\subseteq A \},$$
$$\edgesp(A)=\{(x,x+e_j) \colon j \in \{1,2,\ldots,d\}, \ \{x,x+e_j\} \cap A \neq \emptyset\},$$
$$\nbhd(A)=\{ x+y \colon \ x\in A, \ y \in \{-1,0,1\}^d\}.$$
\end{definition}
So, $\edges(A)$ is the set of positively oriented edges
whose both endpoints are in $A$, $\edgesp(A)$ is the set of
positively oriented edges whose at least one endpoint is in
$A$, and $\nbhd(A)$ is the neighbourhood of $A$ (in the
sup-norm).

\begin{definition}
We  say that a subset $C$ of $\Z^d$ is a \textit{cube} if $C$
is of the form
$$\{n_1,n_1+1,\ldots,n_1+k_1\}\times \ldots \times \{n_d,n_d+1,\ldots, n_d+k_d\}$$
for some $n_1, \ldots, n_d, k_1, \ldots, k_d \in \Z$ with
$k_1,\ldots,k_d\ge 0$.  By the \textit{upper face} of $C$ we
mean
$$\{n_1,n_1+1,\ldots,n_1+k_1\}\times \ldots \times \{n_d+k_d\}.$$
\end{definition}



\begin{definition}
For any $x_1, x_2, x_3, x_4 \in \Z^d$ which are consecutive
vertices of a unit square, and a real number $s$ we define a
$0$-flow $\vartheta^{x_1,x_2,x_3,x_4}_s$ by the following
formula:

$$\vartheta^{x_1,x_2,x_3,x_4}_s(y,z) = \begin{cases} 
s & \text{ for } (y,z)\in\{(x_1,x_2), (x_2,x_3), (x_3,x_4),
(x_4,x_1)\}, \\ -s & \text{ for } (y,z)\in\{(x_2,x_1),
(x_3,x_2), (x_4,x_3), (x_1,x_4)\}, \\ 0 & \text{ otherwise.}
\end{cases}
$$  
\end{definition}
That is, $\vartheta^{x_1,x_2,x_3,x_4}_s$ is a flow sending
$s$ units through the path $x_1\to x_2\to x_3\to x_4\to
x_1$.

Note that if $\varphi \colon G \to \R$ is an $f$-flow and
$s=\varphi(x_1,x_4)-\lfloor \varphi(x_1,x_4)\rfloor$ then
$\psi=\varphi+\vartheta_s^{x_1,x_2,x_3,x_4}$ is an $f$-flow
such that $|\varphi-\psi|<1$ and $\psi(x_1,x_4)$ is an
integer.

We will now prove a couple of lemmas stating that one can modify a flow
so that it becomes integer-valued on certain sets of edges.

\begin{lemma}\label{lem:flow1} Let $f \colon \Z^d \to
\R$. Let $\varphi \colon G \to \R$ be a bounded
$f$-flow. Let
$$C=\{n_1,n_1+1,\ldots,n_1+k_1\}\times \ldots \times \{n_{d-1},n_{d-1}+1,\ldots, n_{d-1}+k_{d-1}\}\times\{n_d,n_d+1\}$$
for some $n_1, \ldots, n_d, k_1, \ldots, k_{d-1} \in \Z$
with $k_1,\ldots,k_{d-1}\ge 0$. Then for every $1\le \ell<d$
there is an $f$-flow $\psi$ such that:
\begin{itemize}
\item $\supp(\varphi - \psi) \subseteq \edges(C)$,
\item for every $x=(x_1,\ldots,x_{d-1},n_d)\in C$ such that
$n_\ell \le x_\ell <n_\ell+k_\ell$ we have
$$\psi(x,x+e_d)\in\Z.$$
\item $|\varphi-\psi|<2$.
\end{itemize}
\end{lemma}

\begin{proof} Without loss of generality we may assume that
$n_1=n_2=\ldots=n_d=0$.

For every $j\leq k_\ell$ define $C_j=\{(x_1,\ldots,x_{d-1},0)\in C
\colon x_\ell=j \}$. We will define a sequence of $f$-flows
$\varphi_0, \varphi_1, \ldots, \varphi_{k_\ell}$ such that
 $$\varphi_j(x,x+e_d)\in\Z\quad
\mbox{and}\quad \supp(\varphi-\varphi_j)\subseteq \edges(C).$$ for all $x \in \bigcup_{i<j} C_i$.

So, let $\varphi_0=\varphi$. Given $\varphi_j$ we define
$\varphi_{j+1}$ in the following way. For every $x\in C_j$
let $\vartheta_x=\vartheta_s^{x,y,z,t}$ where $y=x+e_\ell$,
$z=y+e_d$, $t=z-e_\ell=x+e_d$ and
$$s=\varphi_j(x,t)-\lfloor\varphi_j(x,t)\rfloor.$$

We define
$$\varphi_{j+1} = \varphi_j+ \sum_{x \in C_j} \vartheta_x.$$
Note that $\supp(\vartheta_x)$ for $x\in C_j$ are disjoint
from $\{(x,x+e_d) \colon x \in \bigcup_{i<j}
C_i\}$. Therefore, $\varphi_{j+1}(x,x+e_d) =
\varphi_j(x,x+e_d) \in \Z$ for $x \in \bigcup_{i<j}
C_i$. Also, the sets $\supp(\vartheta_x)$ are pairwise
disjoint for $x\in C_j$, and therefore, by definition of
$\varphi_{j+1}$ we have for $x \in C_j$
$$\varphi_{j+1}(x,x+e_d)=\varphi_j(x,x+e_d)+\vartheta_x(x,x+e_d)=\lfloor \varphi_j(x,x+e_d) \rfloor \in \Z.$$
It is also clear that $\supp(\vartheta_x) \subseteq \edges(C)$, so
$$\supp(\varphi - \varphi_{j+1}) \subseteq \supp(\varphi-\varphi_j)\cup \bigcup_{x\in C_j} \supp(\vartheta_x) \subseteq \edges(C).$$
Therefore $\varphi_{j+1}$ satisfies all required properties.

We put $\psi=\varphi_{k_\ell}$. It remains to check that
$|\varphi - \psi|<2$. This is because $\psi = \varphi +
\sum_{j=0}^{k_\ell-1} \sum_{x\in C_j} \vartheta_x$,
$|\vartheta_x|<1$ and for every edge $(y,z)$ there are at
most two $x\in\bigcup_{j<k_\ell} C_j$ for which
$\vartheta_x(y,z)\neq 0$.
\end{proof}

\begin{lemma}\label{lem:flow2} Let $f \colon \Z^d \to
\R$. Let $\varphi \colon G \to \R$ be a bounded
$f$-flow. Let
$$C=\{n_1,n_1+1,\ldots,n_1+k_1\}\times \ldots \times \{n_{d-1},n_{d-1}+1,\ldots, n_{d-1}+k_{d-1}\}\times\{n_d,n_d+1\}$$
for some $n_1, \ldots, n_d, k_1, \ldots, k_{d-1} \in \Z$
with $k_1,\ldots,k_{d-1}\ge 0$. Then there is an $f$-flow
$\psi$ such that:
\begin{itemize}
\item $\supp(\varphi-\psi)\subseteq \edges(C)$,
\item if $x=(x_1,\ldots,x_{d-1},n_d)\in
C\setminus\{(n_1+k_1,n_2+k_2,\ldots,n_{d-1}+k_{d-1},n_d)\}$,
then
$$\psi(x,x+e_d)\in \Z,$$
\item $|\varphi-\psi|<2d$.
\end{itemize}
\end{lemma}

\begin{proof} Without loss of generality we may assume that
$n_1=n_2=\ldots=n_d=0$.

Define
$$C_j = \{k_1\}\times \ldots \times \{k_{\ell-1}\}\times \{0,1,\ldots,k_\ell\}\times \ldots\times\{0,1,\ldots, k_{d-1}\}\times\{0,1\}$$
and
$$D_j=\{(x_1,\ldots, x_{d-1},0) \colon (x_1,\ldots,x_j) \neq (k_1,\ldots, k_j)\}.$$

By induction, construct $f$-flows
$\varphi_0,\varphi_1,\ldots,\varphi_{d-1}$ such that
\begin{itemize}
\item[(i)] $\supp(\varphi-\varphi_j) \subseteq \edges(C)$,
\item[(ii)] $\varphi_j(x,x+e_d) \in \Z$ for every $x\in D_j$,
\item[(iii)] $|\varphi_j-\varphi_{j-1}|<2$.
\end{itemize}

We define $\varphi_0=\varphi$. Given $\varphi_{j-1}$, we
obtain $\varphi_j$ by applying Lemma \ref{lem:flow1} for
$\varphi_{j-1}$, $f$, $\ell=j$ and $C_j$. Then $\varphi_j$
satisfies (i) as
$$\supp(\varphi-\varphi_j) \subseteq \supp(\varphi-\varphi_{j-1}) \cup \supp(\varphi_{j-1}-\varphi_j) \subseteq \edges(C) \cup \edges(C_j) = \edges(C).$$
For (ii) observe that
$$D_j=D_{j-1} \cup \{(k_1,\ldots,k_{j-1},x_j,\ldots,x_{d-1},0)\in C \colon x_j < k_j\}.$$
By Lemma \ref{lem:flow1}, $\varphi_j$ agrees with
$\varphi_{j-1}$ on $\{(x,x+e_d) \colon x \in D_{j-1}\}$,
thus $\varphi_j(x,x+e_d)\in \Z$ for $x\in
D_{j-1}$. Moreover, $\varphi_j(x,x+e_d)\in \Z$ for $x\in
D_j\setminus D_{j-1}$ again by Lemma \ref{lem:flow1}. Also (iii) is immediate by Lemma
\ref{lem:flow1}. Therefore $\varphi_j$ satisfies the required
properties.

We define $\psi = \varphi_{d-1}$. By construction, $\psi$
satisfies the first two conditions. For the third condition
note that
$$|\varphi-\psi| \le \sum_{j=1}^{d-1} |\varphi_j - \varphi_{j-1}| < 2d.$$
\end{proof}

\begin{lemma}\label{comblemma} 
Let $C$ be a cube. Let $\mathcal C$ be a
collection of cubes such that:
\begin{itemize}
\item $\nbhd(C') \subseteq C$ for every $C'\in \mathcal C$,
\item $\nbhd(C') \cap \nbhd(C'') = \emptyset$ for every distinct
$C', C''\in \mathcal C$.
\end{itemize} 
Write $$E=\edgesp(C)\setminus \bigcup \{\edges(\nbhd(C')) \colon C'\in\mathcal C\}.$$

Let $f \colon \Z^d \to \Z$. Let $\varphi
\colon G \to \mathbb R$ be a bounded $f$-flow. Then there
exists an $f$-flow $\psi \colon G \to \mathbb R$ such that:
\begin{itemize}
\item $\supp(\varphi-\psi)\subseteq \edges(\nbhd(C))$,
\item $\supp(\varphi-\psi)$ is disjoint from $\edgesp(C')$ for
every $C'\in \mathcal C$,
\item $\psi(e)$ is integer for every edge $e \in E$,
\item $|\varphi-\psi|<6d$.
\end{itemize}
\end{lemma}

\begin{proof}

Without loss of generality we may assume that
$$C=\{1,2,\ldots,k_1\}\times\ldots\times\{1,2,\ldots,k_d\}$$
for some positive integers $k_1,\ldots,k_d$. Then
$$\nbhd(C)=\{0,1,\ldots,k_1+1\}\times\ldots\times\{0,1,\ldots,k_d+1\}.$$

For any $0\le k\le k_d$ let $H_k = \Z^{d-1} \times
\{k\}$. Let
$$E_{2k}=\{(x,x+e_d) \in E \colon x\in H_k \}$$
be the set of vertical edges from $E$ having their starting
point in $H_k$ and
$$E_{2k+1}=\{(x,x+e_j) \in E \colon x \in H_k, j<d\}$$
be the set of edges from $E$ having both endpoints in $H_k$.

We construct a sequence $\varphi_0, \varphi_1, \ldots,
\varphi_{2k_d}$ of $f$-flows so that
\begin{itemize}
\item $\supp(\varphi-\varphi_k)\subseteq \edges(\nbhd(C))$ for every $0
\le k \le 2k_d$,
\item $\supp(\varphi-\varphi_k)$ is disjoint from $\edgesp(C')$
for every $C'\in \mathcal C$ and $0\le k\le 2k_d$,
\item $\varphi_k(y,z)$ is integer for every $0\le k \le
2k_d$ and $(y,z) \in \bigcup_{i\le k} E_i$.
\end{itemize} In the end we put $\psi=\varphi_{2k_d}$.

\begin{figure*}[!t]\label{fig:nondiag}
	\centering
	\includegraphics[width=0.75\textwidth]{fig-3.mps}
	\caption{Construction of $\varphi_{2k+1}$.}
	\label{f:1}
\end{figure*}

To define $\varphi_0$ we use Lemma $\ref{lem:flow1}$ for
$\varphi$, $f$, $\ell=1$, and the cube
$$\{1,2,\ldots,k_1+1\}\times \{1,2,\ldots,k_2\} \times \{1,2,\ldots,k_3\} \times \ldots \times \{1,2,\ldots,k_{d-1}\} \times \{0,1\}.$$

Suppose that $\varphi_{2k}$ is defined. Now we define
$\varphi_{2k+1}$ (cf. Fig. \ref{f:1}). For every edge
$(x,y)\in E_{2k+1}$ let $z=y+e_d$, $t=x+e_d$,
$s=-\varphi_{2k}(x,y)+\lfloor \varphi_{2k}(x,y) \rfloor$ and
$\vartheta_{(x,y)}=\vartheta_s^{x,y,z,t}$.

Define $\varphi_{2k+1}=\varphi_{2k}+\sum \vartheta_{(x,y)}$
where the summation goes over all $(x,y)\in E_{2k+1}$.

Note that $\varphi_{2k+1}$ assumes integer values on all
$(x,y)\in E_{2k+1}$. Indeed, if $(x',y')\in E_{2k+1}$ is
distinct from $(x,y)$ then $\vartheta_{(x',y')}(x,y)=0$ and
so
$$\varphi_{2k+1}(x,y)=\varphi_{2k}(x,y)+\vartheta_{(x,y)}(x,y)=\lfloor \varphi_{2k}(x,y) \rfloor \in \Z.$$
Moreover, by definition, $\varphi_{2k+1}$ agrees with
$\varphi_{2k}$ on $\bigcup_{i\le 2k} E_i$. It follows that
$\varphi_{2k+1}$ is integer-valued on $\bigcup_{i\le 2k+1}
E_i$.

Since for every $(x,y)\in E_{2k+1}$ we have $\supp(\vartheta_{(x,y)})
\subseteq \edges(\nbhd(C))$ and $\supp(\vartheta_{(x,y)}) \cap \edgesp(C') =
\emptyset$ for every $C' \in \mathcal C$, and $\varphi_{2k}$
satisfies these as well by inductive hypothesis, we see that
$\varphi_{2k+1}$ also has these properties.

Thus $\varphi_{2k+1}$ is as required.

Now suppose that $\varphi_{2k+1}$ is defined. We construct
$\varphi_{2k+2}$ (cf. Fig. \ref{f:2}). Let $D=\{x
\colon (x,x+e_d)\in E_{2k+2}\}$. Note that every $x\in D$ is
either an element of $C\setminus \bigcup \{\nbhd(C')\colon C'\in
\mathcal C\}$ or lies on the upper face of some cube
$\nbhd(C')$ for $C'\in \mathcal C$. We also note that if $C'\in
\mathcal C$ then the upper face of $\nbhd(C')$ is either
contained in $D$ or disjoint from $D$. So, let $C_1, C_2,
\ldots C_n$ be all elements of $\mathcal C$ such that the
upper faces $D_1, D_2, \ldots, D_n$ of $\nbhd(C_1), \nbhd(C_2),
\ldots, \nbhd(C_n)$ are subsets of $D$.

Let $(x,x+e_d)\in E_{2k+2}$. Then either $x \in D_j$ for
some $j\leq n$ or $x \in D \setminus \bigcup_{j\le n} D_j$.

First we deal with the case $x \in D \setminus \bigcup_{j\le n} D_j$. Then $(x-e_d,x)\in E_{2k}$ and $(x,x+e_i),
(x-e_i,x)\in E_{2k+1}$ for every $1 \le i \le d-1$. By the
inductive hypothesis
$$\varphi_{2k+1}(x,x\pm e_1), \varphi_{2k+1}(x,x\pm e_2), \ldots, \varphi_{2k+1}(x,x\pm e_{d-1}), \varphi_{2k+1}(x,x-e_d)\in\Z.$$ 
Since $f(x)\in \Z$ and
$$f(x)=\sum_{i=1}^{d} \varphi_{2k+1}(x,x\pm e_i),$$
it follows that $\varphi_{2k+1}(x,x+e_d)\in \Z$.

\begin{figure*}[!t]\label{fig:nondiag}
	\centering
	\includegraphics[width=0.75\textwidth]{fig-4.mps}
	\caption{Construction of $\varphi_{2k}$.}
	\label{f:2}
\end{figure*}

Next we deal with the case $x \in D_j$ for some $j\leq
n$. Each $D_j, j\leq n$ is dealt with
separately. For every $j\leq n$ we obtain an $f$-flow
$\varphi'_j$ by applying Lemma \ref{lem:flow2} for
$\varphi_{2k+1}$, $f$ and the cube
$$D'_j=D_j \cup (D_j+e_d) = \{n'_1,\ldots,n'_1+k'_1\}\times \ldots \times \{n'_{d-1},\ldots,n'_{d-1}+k'_{d-1}\}\times \{n'_d,n'_d+1\}.$$ 
Then $\varphi'_j$ agrees with $\varphi_{2k+1}$ outside of
$\edges(D'_j)$, and $\varphi'_j$ is also integer-valued on
all edges of the form $(x,x+e_d)$ with
$x\in D_j \setminus \{x'\}$, where
$$x'=(n'_1+k'_1,
n'_2+k'_2,\ldots,n'_{d-1}+k'_{d-1},n'_d).$$
The only problematic edge is the one $(x',x'+e_d)$ 
We claim that $\varphi'_j(x',x'+e_d)$ is integer as well.

Indeed, observe that
$$\sum_{x\in \nbhd(C_j)} f(x) = \sum_{(x,y)\in E, x\in \nbhd(C_j),y \notin \nbhd(C_j)} \varphi'_j(x,y)$$
Since
$f(x)\in \Z$ for every $x$ and, by the properties of
$\varphi'_j$, we have that $\varphi'_j(x,y) \in \Z$ for all
$(x,y)\neq (x',x'+e_d)$ with $x \in \nbhd(C_j)$ and $y \notin
\nbhd(C_j)$, it follows that $\varphi'_j(x', x'+e_d) \in \Z$
as well.

We define $\varphi_{2k+2}$ by the formula
$$\varphi_{2k+2}(x,y) = \begin{cases} \varphi'_j(x,y) & \text{ if } (x,y) \in \edges(D'_j) \text{ or } (y,x) \in \edges(D'_j) \text{ for some } j, \\
\varphi_{2k+1}(x,y) & \text{ otherwise.}
\end{cases}$$ $\varphi_{2k+2}$ is well-defined because
$\edges(D_j')$ are pairwise disjoint. By definition, it is
integer-valued on $\bigcup_{i\le 2k+2} E_i$, and the
conditions on $\supp(\varphi-\varphi_{2k+2})$ are clearly
satisfied.  Thus $\varphi_{2k+2}$ is as required.

We put $\psi = \varphi_{2k_d}$. It remains to check that
$|\varphi - \psi|<6d$. This follows from the fact that the
value on every edge was modified at most three times by at
most $2d$.
\end{proof}

\section{Measurable bounded $\Z$-flows a.e.}

In this section we show how to turn a measurable bounded real-valued
flow into a measurable bounded integer-valued flow on a set
of measure 1. We only use Lemma \ref{comblemma} proved in
the previous section and the Gao--Jackson tiling theorem for
actions of $\Z^d$.

Suppose $\Z^d\curvearrowright (X,\mu)$ is a free pmp
action. We follow the notation from the previous section in
the context of the action.

\begin{definition}
  We say that a finite subset of $X$ is a \textit{cube} if
  it is of the form
  $$(\prod_{i=1}^d k_i)\cdot x=(\{0,1,\ldots,k_1\}\times \ldots \times \{0,1,\ldots, k_d\})\cdot x$$
  for some positive integers $k_1, \ldots, k_d$ and
  $x \in X$.  We refer to the numbers $k_1,\ldots,k_d$ as to
  the lengths of the sides of the cube.  A family of cubes
  $\{(\prod_{i=1}^d k_i(x))\cdot x: x\in C\}$ is
  \textit{Borel} if the set $C$ is Borel and the functions
  $k_i$ are Borel. A family of cubes $\{C_x: x\in C\}$ is a
  \textit{tiling} of $X$ if it forms a partition of $X$.
\end{definition}


\begin{definition}
  Let $\mathcal C \subseteq [X]^{<\infty}$ be a collection of
  cubes. We say that it is \textit{nested} if for every
  distinct $C, C' \in \mathcal C$:
  \begin{itemize}
  \item if $C \cap C' = \emptyset$ then
    $\nbhd(C) \cap \nbhd(C')=\emptyset$,
  \item if $C \cap C' \neq \emptyset$ then either
    $\nbhd(C) \subseteq C'$ or $\nbhd(C') \subseteq C$.
  \end{itemize}
\end{definition}

\begin{definition}
  Given a cube of the form
  $$C=\{(n_1,\ldots,n_d)\cdot x: 0\leq n_i\leq N_i\},$$ by its
  \textit{interior} we mean the cube
  $$\inter C=\{(n_1,\ldots,n_d)\cdot x: 1\leq n_i\leq
  N_i-1\}$$
  and its \textit{boundary} is $$\bd C=C\setminus \inter C.$$
\end{definition}

\begin{lemma}\label{nested}
  Suppose $\Z^d\curvearrowright (X,\mu)$ is a free pmp
  action. Then there is a sequence of familes $F_n$ of
  cubes such that each $F_n$ consists of disjoint cubes,
  $\bigcup F_n$ is nested and covers $X$ up to a set of
  measure zero.
\end{lemma}
\begin{proof}

  If $S$ and $T$ are families of sets, define
  $$S \tnij T = \{C \cap C' \colon C \in S, \ C' \in T,\ C
  \cap C' \neq \emptyset\}.$$
  Note that
  $\bigcup (S \tnij T) = \left(\bigcup S\right) \cap
  \left(\bigcup T\right)$.
  Also note that if $S$ and $T$ are families of cubes then
  $S \tnij T$ is a family of cubes as well. We also write
  $\Int S = \{ \Int C \colon C \in S\}$ and $\Int^k$ for the
  $k$-th iterate of $\Int$.

  Use the Gao--Jackson theorem \cite{gao.jackson} to obtain
  a sequence of partitions $S_1, S_2, \ldots$ of $X$ so that
  $S_n$ consists of cubes with sides $n^3$ or
  $n^3+1$. Define $S_n^1= \Int S_n$ and
  $S_n^k = S_n^{k-1} \tnij \Int^k S_{n+k}$ for $k>1$. Note
  that each $S_n^k$ consists of pairwise disjoint cubes.

  Define
  $$F_n = \liminf_m S_n^m = \{C \colon \exists m_0 \forall m\ge m_0 \ C \in S_n^m \}.$$ 
  Note that if $C \in F_n$ then there exist unique cubes
  $C_n \in S_n, C_{n+1} \in S_{n+1}, \ldots$ such that
  $C = \bigcap_{k\ge 0} \Int^{k+1} C_{n+k}$. Also note that
  $\bigcup F_n = \bigcap_{k=0}^\infty \bigcup \Int^{k+1}
  S_{n+k}$.

  We claim that $F=\bigcup_n F_n$ is nested and covers a set
  of measure $1$.

  For nestedness, consider cubes $C, C' \in F$. Then
  $C \in F_n$, $C' \in F_m$ for some $n,m$. We may assume
  that $m\ge n$. Write
  $C = \bigcap_{k\ge 0} \Int^{k+1} C_{n+k}$ and
  $C' = \bigcap_{k\ge 0} \Int^{k+1} C_{m+k}$ with
  $C_k, C_k' \in S_k$.

  If $m=n$ and $C_k=C_k'$ for all $k\ge m$ then $C=C'$.

  If $m>n$ and $C_k=C_k'$ for all $k\ge m$ then
  $$C \subseteq \bigcap_{k\ge m} \Int^{k-n+1} C_k = \bigcap_{k\ge m} \Int^{k-n+1} C_k' \subseteq \bigcap_{k\ge m} \Int^{k-m+2} C_k,$$ so
  $\nbhd(C) \subseteq \bigcap_{k\ge m} \Int^{k-m+1} C_k =
  C'$.

  If $C_k \neq C_k'$ for some $k\ge m$ then
  $C_k \cap C_k' = \emptyset$. Note that
  $C \subseteq \Int^{k-n+1} C_k \subseteq \inter C_k$ so
  $\nbhd(C) \subseteq C_k$. Similarly,
  $\nbhd(C') \subseteq C_k'$. Since $C_k, C_k' \in S_k$ are
  disjoint, $C$ and $C'$ are disjoint.

  This shows that $F$ is nested.

  We will prove now that $\mu(\bigcup F) = 1$.

  For a cube $C$ let $x_C$ to be the point $x \in X$ such
  that $C=\left(\prod_{i=1}^d [0,n_i]\right) \cdot x_C$. For
  a positive integer $n$ write
  $X_n = \{x_C \colon C \in S_n\}$. Note that for any
  $0\le k<n$
  $$\mu\left(\bigcup \Int^k S_n\right) \ge (n^3-2k)^d \mu(X_n) \ge \frac{(n^3-2k)^d}{(n^3+1)^d} = \left(1-\frac{2k+1}{n^3+1}\right)^d \ge 1-d\cdot \frac{2k+1}{n^3+1}.$$

  Since
  $\bigcup F_n = \bigcap_{k=0}^\infty \bigcup \Int^{k+1}
  S_{n+k}$,
  we have
  $$\mu\left(X\setminus \bigcup F_n\right) \le
  \sum_{k=0}^\infty \mu\left(X \setminus \bigcup \Int^{k+1}
    S_{n+k}\right) \le d \cdot \sum_{k=0}^\infty
  \frac{2k+3}{(n+k)^3+1} \le d \cdot \sum_{k=n}^\infty
  \frac{3}{k^2}.$$ This implies
  that 
  $$\mu\left(X\setminus \bigcup F\right) = \lim_{n\to\infty}\mu\left(X\setminus \bigcup F_n\right) = 0.$$
  Hence $\mu\left(\bigcup F\right)=1$.

\end{proof}

Marks and Unger \cite[Lemma 5.4]{mu} showed that for every $d\geq 2$, any Borel,
bounded real-valued flow on the Schreier graph of a free
Borel action of $\Z^d$ can be modified to a bounded Borel
integer-valued flow. Below we provide a short proof for the
case $d=1$ and additionally an independent proof (based on
Lemma \ref{comblemma}) for
$d\geq2$ in the case of a pmp action where we consider flows
defined a.e.

\begin{prop}\label{integerflow}
   Suppose $\Z^d\curvearrowright (X,\mu)$ is a free pmp
   action and $G$ is its Schreier graph. Let $f:X\to\Z$ be a bounded measurable function. For every
   measurable $f$-flow $\varphi:G\to\R$, there
   exists a measurable bounded $\psi:G\to\Z$ such that:
   \begin{itemize}
   \item $\psi$ is an $f$-flow $\mu$-a.e.,
   \item $|\psi|\leq|\varphi| +12d$.
   \end{itemize}
\end{prop}
\begin{proof}
  First we deal with the case $d=1$. In that case for every
  $e\in G$ we simply put
  $\psi(e)=\lfloor\varphi(e)\rfloor$. Note that since $G$ is
  a graph of degree $2$, for every
  $x\in X$, the fractional parts of the two edges which
  contain $x$ are equal because $f$ is integer-valued. Thus,
  $\psi$ is also an $f$-flow.

  Now suppose $d\geq2$. By Lemma \ref{nested}, there exists
  an invariant subset
  $X'\subseteq X$ of measure $1$ and a sequence of families $F_n$ of
  cubes such that $\bigcup_{n\in\N} F_n$ is nested, each $F_n$ consists of disjoint cubes,
  $\bigcup_{n\in\N} F_n$ covers $X'$. By induction on
  $n$ we construct measurable $f$-flows $\varphi_n$ such that
  $\varphi_0=\varphi$ and
  \begin{itemize}
  \item $\supp(\varphi_{n+1}-\varphi_n)\subseteq\bigcup\{\edges(\nbhd(C)): C\in F_n\}$,
  \item $\varphi_{m}=\varphi_{n+1}$ for every $m>n$ on every $\edgesp(C)$
    for $C\in F_n$,
  \item $|\varphi_n|\leq|\varphi|+12d$.
  \end{itemize}
  Given the flow $\varphi_n$ we apply Lemma \ref{comblemma} on each cube $C\in
  F_n$ to obtain the flow $\varphi_{n+1}$. The bound on
  $\varphi_n$ follows from the fact that the value of the
  flow on each edge
  is changed at most twice by at most $6d$ along this construction.

  The sequence $\varphi_n$ converges pointwise on the
  edges of $X'$ to a
  measurable $f$-flow
  $\varphi_\infty$, which is integer-valued on all edges in
  $X'$ except possibly for the edges in $\bd C $ for cubes $C\in\bigcup_n
  F_n$. However, the family $\{\bd C: C\in\bigcup_n
  F_n\}$ consists of pairwise disjoint finite sets. By the
  integral flow theorem for finite graphs, we can further
  correct $\varphi_\infty$ on each of these finite subgraphs
  without changing the bound $|\varphi|+12d$
  to obtain a measurable integer-valued $f$-flow $\psi$, which is equal to $\varphi_\infty$ on
  all edges from $G\setminus \bigcup\{\edges(\bd C): C\in\bigcup_{n\in\N}
  F_n\}$.
  \end{proof}

\section{Hall's theorem}
\label{sec:halls-threorem}

In this section we prove Theorem \ref{hall}. The proof of
(1)$\Rightarrow$(2) is based on an idea of Marks and Unger \cite{mu}.

\begin{proof}[Proof of Theorem \ref{hall}]
  (2)$\Rightarrow$(3) is obvious.

  (3)$\Rightarrow$(1) is true for every finitely generated
  group $\Gamma$. In general, if $A$ and $B$ are
  $\Gamma$-equidecomposable, and the group elements used in
  the decomposition are $\gamma_1,\ldots,\gamma_n$, then $A$ and $B$
  satisfy the $k$-Hall condition for $k$ greater than the
  word lengths of the group elements $\gamma_1,\ldots,\gamma_n$. If
  $X'\subseteq X$ is a set of measure $1$ such that
  $A\cap X'$ and $B\cap X'$ are $\Gamma$-equidecomposable, then
  $A\cap X'$ and $B\cap X'$ satisfy the $k$-Hall condition.

  (1)$\Rightarrow$(2). Without loss of generality assume
  that the $k$-Hall condition and $\Gamma$-uniformity is
  satisfied everywhere
 Let
  $\Gamma=\Z^d\times \Delta$ where $\Delta$ is a finite
  group and $d\geq 0$.

  If $d=0$, then the group $\Gamma$ is finite and the action
  has finite orbits (the discrepancy condition trivializes
  and we do not need to use it).  On each orbit the Hall
  condition is satisfied, so on each orbit there exists a
  bijection between $A$ and $B$ on that orbit. Thus, the
  sets $A$ and $B$ are $\Delta$-equidecomposable using a
  Borel choice of bijections on each orbit separately.

  Thus, we can assume for the rest of the proof that
  $d\geq 1$. Since $\Delta$ is finite, we can
  quotient by its action and get a standard Borel space
  $X'=X\slash \Delta$ with the probability measure induced by
  the quotient map $\pi:X\to X'$. We then have a free pmp
  action of $\Z^d\curvearrowright X'$. Consider the function
  $f:X'\to \Z$ defined by
  $$f(x')=|A\cap \pi^{-1}(\{x'\})|-|B\cap
  \pi^{-1}(\{x'\})|.$$
  Note that $f$ is bounded by $|\Delta|$. Using Proposition
  \ref{real} and Proposition \ref{integerflow} we get an invariant
  subset $Y'\subseteq X'$ of measure $1$ and an integer-valued
  measurable $f$-flow $\psi$ on the edges of the Schreier
  graph $G$ of $\Z^d\curvearrowright Y'$ on $Y'$ such that
  $|\psi|\leq |\Delta|\,d^k+12d$. Again, without loss of
  generality, we can assume $Y'=X'$ by replacing $X$ with
  $Y=\pi^{-1}(Y')$, if needed.

  Note that there exists a constant $r$, depending only on $d$ such that for every
  tiling of $\Z^d$ with cubes with sides $n$ or $n+1$, every
  cube is adjacent to at most $r$ many other cubes in the tiling.

  Note that $\Gamma$-uniformity implies that for every set
  $D$ such that $D=D'\times\Delta$ where $D'$ is a cube
  with sides $n$ or $n+1$ we have $|A\cap D|,|B\cap
  D|\geq  c n^d|\Delta|$. Let $n$ be such
  that 
  $c n^d|\Delta|
\geq r
  (n+1)^{d-1} (|\Delta|\,d^k+12d)$.

  Using the Gao--Jackson theorem \cite{gao.jackson}, find a
  Borel tiling $T'$ of $X'$ with cubes of sides $n$ or
  $n+1$. Pulling back the tiling to $X$ via $\pi$, we get a
  Borel tiling $T$ of $X$ with cubes of the form
  $D=(D'\times \Delta)\cdot x$ where $D'$ has sides of length $n$ or
  $n+1$. 
  Note that for every tile $D$ in $T$ we have
  \[|A\cap D|, |B\cap D|\geq r
  (n+1)^{d-1} (|\Delta|\,d^k+12d).\label{unif}\tag{$*$}\]
  Let $H$ be the graph on $T$ where two cubes are connected
  with an edge if they are adjacent and similarly let $H'$
  be the graph on $T'$ with two cubes connected with an edge
  if they are adjacent. We have two functions $F':T'\to\Z$
  defined as $F'(C)=\sum_{x'\in C}f(x')$ and $F:T\to\Z$ defined as
  $$F(C)=|A\cap C|-|B\cap C|.$$ 

Define an $F'$-flow $\Psi'$ on $H'$ as
$\Psi'(C,D)=\sum_{(x_1',x_2')\in G, x_1'\in C, x_2'\in
  D}\psi(x_1',x_2')$ and let $\Psi$ be an $F$-flow on $H$
obtained by pulling back $\Psi'$ via $\pi$.
 Note that any adjacent cubes in $T'$ are connected by at
 most $(n+1)^{d-1}$ edges, so both $\Psi$ and $\Psi'$ are
 bounded by $|\Psi|,|\Psi'|\leq (n+1)^{d-1}(|\Delta|\,d^k+ 12d)$. 

  Note that each vertex in $H'$ has degree at most 
$r$ and the same is true in $H$.

Thus, by (\ref{unif}), for each
  $C\in T$ and $D\in T$ which are connected with an edge in
  $H$, we can find pairwise disjoint sets
  $A(C,D),B(C,D)\subseteq C$ of size at least
  $(n+1)^{d-1}(|\Delta|\,d^k+ 12d)$ such that
  $A(C,D)\subseteq A\cap C$, $B(C,D)\subseteq B\cap C$.

  Now, the function which witnesses the equidecomposition is
  defined in two steps. First, we can find a map $g \colon \dom(g) \to B$ such that $\dom(g)\subset \bigcup_{(C,D)\in H} A(C,D)$, for any two neighbouring cubes $C,D\in T$ satisfying 
  $\Psi(C,D)>0$ we have $|\dom(g) \cap A(C,D)| = \Psi(C,D)$ and the points in $\dom(g) \cap A(C,D)$ are mapped injectively to $B(C,D)$. Note that for any cube $C\in T$ the set $C \cap (A\setminus \dom(g))$ contains as many points as the set $C \cap (B \setminus \im(g))$. Hence, one can extend $g$ to a funcion $g' \colon A \to B$ in such a way that its restriction to $C \cap (A\setminus \dom(g))$ is a bijection onto $C \cap (B \setminus \im(g))$ for any cube $C\in T$. Since $\psi$ and
  hence $\Psi'$ and $\Psi$ are measurable, the function $g'$ can be chosen to be measurable and it moves points
  by at most $2(|\Delta|+(n+1)^d)$ in the Schreier graph distance. Thus,
  $g'$ witnesses that $A$ and $B$ are
  equidecomposable using measurable pieces.
\end{proof}

\section{Measurable circle squaring}
\label{sec:meas-circle-squar}

In this section we comment on how Corollary \ref{wniosek} follows from
Theorem \ref{hall}. We use an argument which appears in a preprint
of Grabowski, M\'ath\'e and Pikhurko \cite{gmp.preprint} and
provide a short proof for completeness.

\begin{lemma}\label{pierwszy}
  Suppose $\Gamma \curvearrowright (X,\mu)$ is a free pmp
  action of a countable group $\Gamma$. If $A,B\subseteq X$
  are $\Gamma$-equidecomposable and $X'\subseteq X$ is
  $\Gamma$-invariant, then $A\cap X'$ and $B\cap X'$ are
  also equidecomposable. If $X'$ is additionally
  $\mu$-measurable and $A$ and $B$ are
  $\Gamma$-equidecomposable using $\mu$-measurable pieces,
  then $A\cap X'$ and $B\cap X'$ are
  $\Gamma$-equidecomposable using $\mu$-measurable pieces.
\end{lemma}
\begin{proof}
  The proof is the same in both cases. Let $A_1,\ldots,A_n$
  and $B_1,\ldots,B_n$ be partitions of $A$ and $B$ such
  that $\gamma_i A_i=B_i$ for some $\gamma_i\in \Gamma$. Put
  $A_i'=A_i\cap X'$ and $B_i'=B_i\cap X'$. Then
  $\gamma_i A_i'=B_i'$, so $A'_i$ and $B_i'$ witness that
  $A\cap X'$ and $B\cap X'$ are equidecomposable.
\end{proof}

\begin{lemma}\label{drugi}
  Let $\mu$ be a probability measure on $X$ and
  $\Gamma \curvearrowright X$ be a Borel pmp action of a countable group $\Gamma$. Suppose $A,B\subseteq X$ are
  $\Gamma$-equidecomposable and there exists a
  measurable set $Y\subseteq X$ of
  measure $1$ such that
  $A\cap Y,B\cap Y$ are equidecomposable using $\mu$-measurable
  piecces. Then $A,B$ are equidecomposable using $\mu$-measurable
  pieces.
\end{lemma}
\begin{proof}
  Write $X'=\bigcap_{\gamma\in \Gamma}\gamma X$. Note that $\mu(X')=1$ and
  $\gamma X'=X'$ for all $\gamma\in \Gamma$.
  By Lemma \ref{pierwszy}, $A'=A\cap X'$ and $B'=B\cap X'$ are
  $\Gamma$-equidecomposable using $\mu$-measurable pieces. Write $X''=
  X\setminus X'$ and note that $\gamma X''=X''$ for all $\gamma\in
  \Gamma$. By the previous lemma again, $A''=A\cap X''$ and
  $B''=B\cap X''$ are $\Gamma$-equidecomposable. However, all pieces in
  the latter decomposition all $\mu$-null, hence $\mu$-measurable. This shows
  that $A=A'\cup A''$ and $B=B'\cup B''$ are
  $\Gamma$-equidecomposable using $\mu$-measurable pieces.

\end{proof}

Finally, we give a proof of Corollary \ref{wniosek}.

\begin{proof}[Proof of Corollary \ref{wniosek}]
  Suppose $\Gamma\curvearrowright (X,\mu)$ is a free pmp
  action of a finitely generated abelian
  group $\Gamma$ and $A$ and $B$ are two measurable
  $\Gamma$-uniform sets which are
  $\Gamma$-equidecomposable. Note that since $\Gamma$ is
  amenable, $A$ and $B$ must have the same measure (see
  \cite[Corollary 10.9]{wagon}). Let
  $\gamma_1,\ldots,\gamma_n$ be the elements of
  $\Gamma$ used in the equidecomposition and let $k$ be bigger
  than the lengths of $\gamma_i$. Then $A$ and $B$ satisfy
  the $k$-Hall condition. In particular, $A$ and $B$ satisfy
  the $k$-Hall condition $\mu$-a.e., so by Theorem \ref{hall} there
  is a $\Gamma$-invariant measurable set $X'\subseteq X$ of measure $1$ such
  that $A\cap X'$ and $B\cap X'$ are
  $\Gamma$-equidecomposable using $\mu$-measurable
  pieces. By Lemma \ref{drugi}, $A$ and $B$ are
  $\Gamma$-equidecomposable using $\mu$-measurable pieces as well.
\end{proof}

\bibliographystyle{plain}
\bibliography{refs}

\end{document}